\documentclass[11pt]{article}
\usepackage{latexsym,amsmath,amsthm,amssymb,amscd,amsfonts}
\usepackage[usenames]{color}

\usepackage{fancyhdr}

\newtheorem{thm}{Theorem}

\newtheorem{rem}[thm]{Remark}

\newtheorem{lem}[thm]{Lemma}

\newtheorem{que}[thm]{Question}

\newtheorem{defn}[thm]{Definition}

\newtheorem{que-ed}{Question-editing}

\def\Ham{\operatorname{Ham}}
\def\Hameo{\operatorname{Hameo}}
\def\PHam{\operatorname{PHam}}
\def\PHameo{\operatorname{PHameo}}
\def\supp{\operatorname{supp}}

\def\Symp{\operatorname{Symp}}
\def\Sympeo{\operatorname{Sympeo}}

\def\Id{\operatorname{Id}}

\begin{document}


\title{Uniqueness of generating Hamiltonians for topological Hamiltonian flows}

\author{Lev Buhovsky$^{1}$, Sobhan Seyfaddini}

%
%

\footnotetext[1]{The first author also uses the spelling ``Buhovski" for his family name.}

\footnotetext[2]{First published in Journal of Symplectic Geometry in Volume 11, Number 1 (2013), published by International Press. Copyright $ \copyright $ 2013 Journal of Symplectic Geometry.}

\maketitle

\begin{abstract}
We prove that a topological Hamiltonian flow as defined by Oh and M\"uller~\cite{OM}, has a unique $L^{(1,\infty)}$ generating topological Hamiltonian function. This answers a question raised by Oh and M\"uller in~\cite{OM}, and improves a result of Viterbo~\cite{V}.
\end{abstract}

\section{Introduction}

Let $(M^{2n},\omega)$ be a symplectic manifold of dimension $2n$, which is closed and connected. Non-degeneracy
of the symplectic form implies that $\omega^n$ is a volume form on $M$.

Throughout the paper we assume that all Hamiltonians are normalized in the following way: given a time dependent
Hamiltonian $ H : [0,1] \times M \rightarrow \mathbb{R} $ we require that $\int_M H(t,x) \omega^n = 0, \forall t\in[0,1]$. For a given open subset $ U \subset M $, we denote by $ \Ham_U(M,\omega) $ the set of all time-$1$ maps of smooth
Hamiltonian flows that coincide with the identity flow on $ M \setminus U $.
We denote by $ C_0^\infty([0,1]\times M) $ the space of all smooth normalized Hamiltonian functions
$ H : [0,1] \times M \rightarrow \mathbb{R} $. The space $ C_0^\infty([0,1]\times M) $ possesses $L^{(1,\infty)}$ norm,
known as the Hofer~\cite{HZ} norm, which is defined as
\[ \Vert H \Vert_{(1,\infty)} = \int_0^1 (\max_x H(t,x) - \min_x H(t,x))\,dt,\]
for $ H \in C_0^\infty([0,1]\times M) $. The completion of $ C_0^\infty([0,1]\times M) $ with respect to the $L^{(1,\infty)}$
norm is denoted by $ L_0^{(1,\infty)}([0,1] \times M) $. We denote by $ C_0^{\infty}(M) $ the space
of smooth functions $ H: M \rightarrow \mathbb{R} $ with $ \int_M H(x) \omega^n = 0 $. We endow $ C_0^\infty(M) $
with the $ L^\infty $ norm: $$ \| H \|_{\infty} = \max_x H(x) - \min_x H(x) .$$ The completion of $ C_0^\infty(M) $ with
respect to the $ L^\infty $ norm is denoted by $ C_0(M) $, and the space $ C_0(M) $ consists of all continuous functions
$ H: M \rightarrow \mathbb{R} $ that satisfy $ \int_M H(x) \omega^n = 0 $.

  We denote by $\PHam(M,\omega)$ the space of smooth Hamiltonian flows. Clearly, given $ \Phi^{t} \in \PHam(M,\omega)$, there exists a unique normalized Hamiltonian $H$, that generates the flow $ \Phi^{t} $.
  The main purpose of this paper is to prove the above uniqueness result for Hamiltonian generators of \emph{topological Hamiltonian paths}, as defined in~\cite{OM}.
  This \emph{``uniqueness of generating Hamiltonians"} turns out to be essential to extending various constructions on spaces $ Ham(M,\omega) $ and $ PHam(M,\omega) $, to the case of topological Hamiltonian flows~\cite{Oh}. For example, uniqueness of the generating Hamiltonian implies that the Oh-Schwarz spectral invariants extend to the space of topological Hamiltonian paths $PHameo(M, \omega)$. Another interesting implication of this uniqueness theorem is that elements of $PHameo(M, \omega)$ corresponding to one-parameter subgroups in the group of Hamiltonian homeomorphisms, $Hameo(M, \omega)$, are generated by autonomous topological Hamiltonians (see the final paragraph in this page for the definitions of $PHameo(M, \omega)$ and $Hameo(M, \omega)$). A corollary of this correspondence is the law of conservation of energy in the present setting.  We refer interested readers to~\cite{Oh} for proofs of the above consequences of the uniqueness theorem.

  The study of continuous symplectic geometry began with the celebrated Eliashberg-Gromov rigidity theorem~\cite{E1, E2, G}, which states that the group $ \Symp(M,\omega)$ of symplectomorphisms of $(M,\omega)$ is $C^0$ closed in the group of diffeomorphisms of $ M $. This theorem motivates the following definition of symplectic homeomorphisms. The group of symplectic homeomorphisms $\Sympeo(M,\omega)$ is defined as the $C^0$ closure of $\Symp(M,\omega)$ in the group of homeomorphisms of $ M $. Extending the notion of Hamiltonian flows turns out to be more complicated.

  In~\cite{OM}, Oh and M\"uller introduce the notions of \emph{topological Hamiltonian paths}, and \emph{Hamiltonian homeomorphisms}. By definition, a continuous path of homeomorphisms $\Phi^t : M \rightarrow M$ is called a \emph{topological Hamiltonian path (or flow), generated by a (topological) Hamiltonian function $H \in L_0^{(1,\infty)}([0,1] \times M)$}, if there exists a sequence of \emph{smooth} Hamiltonian flows, $\Phi^t_{H_i}$, with generating Hamiltonians $H_i\in C_0^\infty([0,1]\times M)$, such that

   $$ \Phi^t = (C^0) \lim_{i \to \infty} \Phi^t_{H_i} ,$$ $$ H = (L^{(1,\infty)}) \lim_{i \to \infty} H_i, $$
that is, the first convergence is in the uniform topology, and the second convergence is in the $L^{(1,\infty)}$ topology. We denote by $\PHameo(M,\omega)$ the space of all pairs $(\Phi^t, H)$ of a topological Hamiltonian flow $ \Phi^t $ and a topological Hamiltonian function $H$, that generates $ \Phi^t $.
The space $ \Hameo(M,\omega) $ of Hamiltonian homeomorphisms is defined to be the set of all time-1 maps of topological Hamiltonian flows.

  \begin{que} \label{Oh's question}
   Does a topological Hamiltonian flow $\Phi^t$ have a unique generating topological Hamiltonian function? In other words, assume we have two (smooth) sequences $(\Phi^t_{H_i}, H_i), (\Phi^t_{K_i}, K_i) \in \PHam(M,\omega)$ satisfying $$ (C^0)\lim \Phi^t_{H_i}=(C^0)\lim \Phi^t_{K_i} = \Phi^t,$$ $$(L^{(1,\infty)})\lim H_i=H,$$
   $$(L^{(1,\infty)})\lim K_i = K.$$  Does this imply $K=H$, as $L^{(1,\infty)}$ functions?
  \end{que}

  This question was raised by Oh and M\"uller~\cite{OM}. The goal of this paper is to give an affirmative answer to the above question.

  Going back to the case of smooth Hamiltonian flows, for given $$ \Phi^t_H, \Phi^t_K \in \PHam(M,\omega) ,$$ generated by smooth Hamiltonians $ H,K $, we have the following well known formulae for the Hamiltonian functions of a composition of flows and an inverse of a flow:

  $\Phi^t_{H}\circ\Phi^t_{K}= \Phi^t_{G}$, where $G = H \# K(t,x):= H(t,x) + K(t,(\Phi^{t}_{H})^{-1}(x))$.

  $(\Phi^t_{H})^{-1} = \Phi^t_{\overline{H}}$, where $\overline{H}(t,x):= -H(t,\Phi^t_{H}(x))$.

  It was shown by Oh and M\"uller~\cite{OM} that these operations admit a natural generalization to the space $ \PHameo(M,\omega) $. It follows that given two pairs $ (\Phi^t, H), (\Phi^t, K) \in \PHameo(M,\omega) $ with common topological Hamiltonian flow, we get the identity flow $ \Id^t = (\Phi^t)^{-1} \circ \Phi^t $ generated by the topological Hamiltonian function $$ \overline{H} \# K(t,x) = - H(t,\Phi^{t}(x)) + K(t,\Phi^{t}(x)).$$

  Hence, question~\ref{Oh's question} simplifies to:

  \begin{que}\label{Main Question}
   Assume we have a sequence of smooth Hamiltonian paths $(\Phi^t_{H_i}, H_i) \in \PHam(M,\omega)$ satisfying
  $$ (C^0)\lim \Phi^t_{H_i}=\Id^t,$$ $$(L^{(1,\infty)})\lim H_i = H.$$  Does this imply $H=0$, as an $L^{(1,\infty)}$ function?
  \end{que}

  In~\cite{V}, Viterbo gives an affirmative answer to the above question assuming $(C^0)\lim H_i=H$. Note that $(C^0)\lim H_i=H$ implies $$(L^{(1,\infty)})\lim H_i=H .$$
The methods employed in this paper are very different than those used in~\cite{V}.

  \begin{rem}

   One can find a sequence of Hamiltonian paths $(\Phi^t_{H_k}, H_k) \in \PHam(M,\omega)$ such that
   $(C^0)\lim \Phi^t_{H_k}=\Id^t$, but the sequence $ ( H_k ) $ does not converge in
   $L_0^{(1,\infty)}([0,1] \times M)$.  To demonstrate this we borrow the following example from~\cite{V}: \\
  Let $ U \subset M $ be a Darboux chart with coordinates $ (x_1, y_1, ..., x_n, y_n) $, such that
  $ 0 = (0,0,...,0,0) \in U $. Let $ r > 0 $ be small enough, such that
  $$ V := B_r(0) = \{  (x_1, y_1, ..., x_n, y_n) \, | \, x_1^2 + y_1^2 + ... + x_n^2 + y_n^2 < r^2 \} \subset U .$$
  Take $H: M \rightarrow \mathbb{R} $ to be any normalized, autonomous and smooth non-zero Hamiltonian supported in $ V $. For any $ k \in \mathbb{N} $,
  define $ H_k : M \rightarrow \mathbb{R} $ by $$ H_k(x_1, y_1, ..., x_n, y_n) = kH(kx_1, ky_1, ..., kx_n, ky_n) $$
  for $ (x_1, y_1, ..., x_n, y_n) \in V_k := B_{\frac{r}{k}}(0) \subset U $, and $ H_k(x) = 0 $ for $ x \in M \setminus V_k $.
  Then the sequence of smooth Hamiltonian paths $(\Phi^t_{H_k})$ $C^0$ converges to $Id^t$, but the sequence of Hamiltonians
  $(H_k)$ diverges.



  \end{rem}

  Section~\ref{S:Main-result} contains the statement of our main result and a formulation of a sequence of lemmata, that are used in its proof. In Section~\ref{S:Proof-of main-result} we present the proof of the main result. Section~\ref{S:Local-uniqueness} studies the local uniqueness for topological Hamiltonian functions and for topological Hamiltonian flows. Here we state and prove the generalization of Theorem 1.3 from~\cite{Oh-local}, to the $ L^{(1,\infty)} $ case. We derive two consequences of this local uniqueness result. First, on any closed symplectic manifold we construct an example of a continuous function, that fails to be a generator of any topological Hamiltonian flow. Second, we give an example of a continuous flow of homeomorphisms on any closed symplectic manifold, which is a $C^0$ limit of smooth Hamiltonian flows, but is not a topological Hamiltonian flow.

\begin{rem}
 All the results in the present paper can be directly generalized to the case of an open symplectic manifold $ (M,\omega) $, where in this case we consider topological Hamiltonian flows that are generated by {\em compactly supported} topological Hamiltonian functions~\cite{OM}.
\end{rem}


\section{Main result} \label{S:Main-result}

In this section we present our main result.

Here's our answer to Question~\ref{Main Question} :

\begin{thm} \label{T:main}
Denote by $ \Id^{t} : M \rightarrow M $ the identity flow.
If we have $ H \in L_0^{(1,\infty)}([0,1] \times M) $, such that $ (\Id^t,H) \in \PHameo(M,\omega) $, then  we have $ H = 0 $ in $ L_0^{(1,\infty)}([0,1] \times M) $.
\end{thm}

We will use the following definition in our proof.

\begin{defn} \label{D:Id-Ham} (Null Hamiltonians) Define
 $$ \mathcal{H}_{0} = \{ H \in  L_0^{(1,\infty)}([0,1] \times M) \, | \, (\Id^t,H) \in \PHameo(M,\omega) \} ,$$ this is the set 
of { \bf null Hamiltonians}. Define
$$ \mathcal{H}^{st}_{0} = \{ H \in \mathcal{H}_{0} \, | \, H \text{ is time independent} \} .$$
An element $H \in L_0^{(1,\infty)}([0,1] \times M)$ is time independent if there exists a representing
function for $H$, as in Lemma~\ref{L:real-analysis} below, that is time independent.
Since $ \mathcal{H}^{st}_{0} $ consists of {\bf time-independent} null Hamiltonians, we identify it with a subset of $ C_0(M) $.

\end{defn}

We divide the proof of Theorem~\ref{T:main} into a sequence of lemmata. Lemma~\ref{L:smooth-case} is the smooth case of Theorem~\ref{T:main}. It has been proven in the past, see e.g.~\cite{OM} or~\cite{HZ}.

\begin{lem} \label{L:smooth-case}
If $ H \in \mathcal{H}_{0} \cap C^{\infty}([0,1] \times M) $, then for all
$ t \in [0,1] $ we have $ H(t,x) \equiv 0 $.
\end{lem}

\begin{lem} \label{L:basic-properties}
The sets $ \mathcal{H}_{0}, \mathcal{H}^{st}_{0} $ have the following properties:

\begin{enumerate}
  \item $\mathcal{H}_{0}$ is closed under the sum operation and the minus operation. In other words, if $ H,K \in \mathcal{H}_{0} $, then $ -H, H+K \in \mathcal{H}_{0} $.
  \item $\mathcal{H}_{0}$ is closed in the $ L^{(1,\infty)} $ topology. $\mathcal{H}_{0}^{st}$ is closed in the $ L^\infty $ topology.
  \item If $ H \in \mathcal{H}_{0} $, then for any smooth increasing function $ \alpha : [0,1] \rightarrow [0,1] $
  the Hamiltonian $ K(t,x) = \alpha'(t) H(\alpha(t),x) $ belongs to $ \mathcal{H}_{0} $ as well.
  \item $\mathcal{H}^{st}_{0}$ is a vector space over $\mathbb{R}$.
  \item If $ H \in \mathcal{H}^{st}_{0} $, then for any $ \Phi \in \Symp(M,\omega) $ we have $ \Phi^{*} H = H \circ
  \Phi \in \mathcal{H}^{st}_{0} $.
\end{enumerate}

\end{lem}

\begin{lem} (Lebesgue's differentiation theorem) \label{L:real-analysis}
Any $ H \in L_0^{(1, \infty)}([0,1] \times M) $
can be represented by a function $ H: [0,1] \times M \rightarrow \mathbb{R} $ (we use the same notation for the function as
well), such that for any $ t \in [0,1] $ we have $ H(t,\cdot) \in C_0(M) $, and such that for any Cauchy sequence
$ ( H_i )_{i=1,2,...} $ in $ C^{\infty}_0([0,1] \times M) $ that represents $ H $, we have
$$ \lim_{i \rightarrow \infty} \| H_i - H \|_{(1,\infty)} = 0 ,$$
where $$ \| H_i - H \|_{(1,\infty)} = \int_0^1 \max_x [H_i(t,x) - H(x,t)] - \min_x [H_i(t,x) - H(x,t)] \, dt. $$
Moreover, almost everywhere in $ t \in [0,1) $ we have $$ \lim_{ h \rightarrow 0^+} \frac{1}{h} \int_{t}^{t+h} \| H(s,\cdot) - H(t,\cdot) \|_{\infty} ds = 0. $$
\end{lem}

\begin{lem} \label{L:reduction}
Let $ H \in \mathcal{H}_{0} $, and denote by the same notation $ H $ its functional representative, as in
Lemma~\ref{L:real-analysis}. Then for almost any $ t \in [0,1] $, the time-independent Hamiltonian $ h(x) = H(t,x) $
lies inside $ \mathcal{H}^{st}_{0} $.
\end{lem}

\begin{lem} \label{L:H^st}
If $ H \in \mathcal{H}^{st}_{0} $, then $ H \equiv 0 $.
\end{lem}

\section{Proofs} \label{S:Proof-of main-result}

\begin{proof} [Proof of Lemma~\ref{L:smooth-case}]
     Assume for a contradiction, that $H$ is not constantly zero. Let $\Phi^t_H$ denote the flow of $H$. Since $H$ is not constantly zero we conclude that $\Phi^T_H$ is not the identity map, for some $ T \in [0,1] $.

     Since $(\Id^t,H) \in \PHameo(M,\omega)$, there exists a smooth sequence $(\Phi^t_{H_i}, H_i) \in \PHam(M,\omega)$ which converges to $(\Id^t, H)$. This implies that $(\Phi^T_{H_i})^{-1}\circ \Phi^T_H $ $ C^0 $ converges to $\Phi^T_H$.  Pick a point $x\in M$ such that $\Phi^T_{H}(x)\neq x$.  There exists a small open neighborhood, $U$ of $x$, which is displaced by $(\Phi^T_{H_i})^{-1}\circ \Phi^T_H$ , for $i$ large enough.  The general energy-capacity inequality~\cite{LM}, implies that the Hofer norm of $(\Phi^T_{H_i})^{-1}\circ \Phi^T_H$ is bounded below by a positive constant, $e(U)$.
     But this norm is bounded from above by $$ \| \overline{H_i} \# H \|_{(1,\infty)} = \Vert -H_i(t,\Phi^t_{H_i}(x)) + H(t,\Phi^t_{H_i}(x))\Vert_{(1,\infty)} =   \| -H_i + H \|_{(1,\infty)} ,$$ what contradicts the $L^{(1,\infty)}$ convergence of $H_i$ to $H$.
\end{proof}

\begin{proof} [Proof of Lemma~\ref{L:basic-properties}] $ $

(1): If $(\lambda^t,H),(\mu^t,K) \in \PHameo(M,\omega)$, then the composition of the pairs, $(\lambda^t\circ\mu^t, H\#K)$, and the inverse flow $ ((\lambda^t)^{-1}, \overline{H}) $ are also in $\PHameo(M,\omega)$~\cite{OM}.  Since $\lambda^t = \mu^t = \Id^t$, we have $H\#K=H+K, \overline{H} = -H $.

(2): This is clear from the definition of $\mathcal{H}_{0} $ and of $ \mathcal{H}^{st}_{0} $.

(3): If $\Phi^t_G$ is a smooth Hamiltonian flow generated by $ G $, then its reparameterized flow $\Phi^{\alpha(t)}_G$ is generated by  $ L(t,x) = \alpha'(t) G(\alpha(t),x) $. If we assume that $ H \in \mathcal{H}_0 $, then there exists a sequence $ H_i(t,x) $ of smooth Hamiltonians, such that we have $ (C^0) \lim
\Phi^{t}_{H_i} = \Id^{t} $, and $ (L^{(1,\infty)}) \lim H_i = H $. Then the reparameterized flows
$\Phi^{\alpha(t)}_{H_i}$ are generated by $ K_i(t,x) = \alpha'(t) H_{i}(\alpha(t),x) $. It is clear, that $$
(C^0) \lim \Phi^t_{K_i} = (C^0) \lim \Phi^{\alpha(t)}_{H_i} = \Id^{t} ,$$ and also $$ (L^{(1,\infty)}) \lim K_i(t,x) = (L^{(1,\infty)}) \lim \alpha'(t) H_{i}(\alpha(t),x) = \alpha'(t) H(\alpha(t),x) .$$ Therefore $ K(t,x) = \alpha'(t) H(\alpha(t),x) \in \mathcal{H}_{0} $.

(4): This follows from the previous results. Suppose $ H \in \mathcal{H}^{st}_{0} $ with the topological Hamiltonian flow $\Phi^t$. For any $ 0 < a < 1 $, apply (3) with $ \alpha(t) = at $ to obtain that $ aH \in \mathcal{H}_{0} $
and hence $ aH \in \mathcal{H}^{st}_{0} $. Then, the case of general $ a \in \mathbb{R} $ follows from (1).

(5): In the smooth case, if $H$ generates the Hamiltonian flow $\Phi^t$, then $\Psi^{*} H$ generates the Hamiltonian flow $\Psi^{-1}\Phi^t\Psi$.  This property extends to topological Hamiltonian flows~\cite{OM}, and hence the result follows.
	
\end{proof}

\begin{proof} [Proof of Lemma~\ref{L:real-analysis}]

Consider a Cauchy sequence $ K_i \in C^{\infty}_0([0,1] \times M) $, $ i=1,2,... $, representing
$ H $. By passing to a subsequence, if necessary, we may assume that $ \| K_{i+1} - K_{i} \|_{(1,\infty)} < \frac{1}{2^i} $
for $ i \geqslant 1 $. Denote $ f_1(t) \equiv 0 $, and $$ f_N(t) := \| K_2(t,\cdot) - K_1(t,\cdot) \|_{\infty} + ... +
\| K_N(t,\cdot) - K_{N-1}(t,\cdot) \|_{\infty} ,$$ for $ N \in \mathbb{N} $, $ N > 1 $, $ t \in [0,1] $. Then $ (f_N) $ is a
non-decreasing sequence of non-negative continuous functions on the interval $ [0,1] $, and we have a bound on the $ L^1 $ norm
$ \| f_N \|_{1} < \sum_{i=1}^{N-1} \frac{1}{2^i} < 1 $, for $ N > 1 $. Therefore it follows that a.e. in $ t \in [0,1] $ there exists a finite
limit $ f(t) := \lim_{N \rightarrow \infty} f_N(t) $, and we have $ f \in L^1[0,1] $ and
$ \lim_{N \rightarrow \infty} \| f - f_N \|_{1} = 0 $. Since a.e. in $ t \in [0,1] $, the sequence $ (f_N(t))_{N=1,2,...} $
converges, and we have $ f_N(t) - f_M(t) \geqslant \| K_N(t,\cdot) - K_M(t,\cdot) \|_{\infty} $ for any $ N > M $, it follows
that for almost any $ t \in [0,1] $, the sequence $ (K_N(t,\cdot))_{N=1,2,...} $ is a Cauchy sequence with respect to the
$ L^\infty $ norm. Therefore for almost any $ t \in [0,1] $, there exists $$ H(t,\cdot) := (L^{\infty}) \lim_{N \rightarrow \infty} K_N(t,\cdot)
\in C_0(M) .$$ For all other $ t \in [0,1] $, define $ H(t,\cdot) \equiv 0 $. Now, for any $ N > M $ and $ t \in [0,1] $
we have $ f_N(t) - f_M(t) \geqslant \| K_N(t,\cdot) - K_M(t,\cdot) \|_{\infty} $, and for almost any $ t \in [0,1] $ we have
$ H(t,\cdot) = (L^{\infty}) \lim_{N \rightarrow \infty} K_N(t,\cdot) $, hence by taking $ N \rightarrow \infty $, for almost
any $ t \in [0,1] $ we obtain $ f(t) - f_M(t) \geqslant \| H(t,\cdot) - K_M(t,\cdot) \|_{\infty} $, for $ M \in \mathbb{N} $.
Finally, since $ f(t) = (L^1) \lim_{M \rightarrow \infty} f_M(t) $, we obtain
$ \lim_{M \rightarrow \infty} \| H - K_M \|_{(1,\infty)} = 0 $.

For any other Cauchy sequence $ H_i \in C^{\infty}_0([0,1] \times M) $, $ i=1,2,... $, representing $ H $, we have
$$ \| H - H_i \|_{(1,\infty)} \leqslant \| H - K_i \|_{(1,\infty)} + \| K_i - H_i \|_{(1,\infty)} $$ for any $ i \in \mathbb{N}
$, and hence we also have $ \lim_{i \rightarrow \infty} \| H - H_i \|_{(1,\infty)} = 0 $.

The second part of the theorem is a reformulation of \emph{Lebesgue's differentiation theorem} for
$ L^1 \text{  maps from [0,1] to the Banach space } C_0(M)$. Consider any Cauchy sequence
$  H_i \in C^{\infty}_0([0,1] \times M) $, $ i=1,2,... $, that represents $ H $. The functions $ H_i $ are continuous and
hence they satisfy $$ \lim_{h \rightarrow 0^+} \frac{1}{h} \int_{t}^{t+h} \| H_i(s,\cdot) - H_i(t,\cdot) \|_{\infty} ds = 0 $$
for all $ t \in [0,1) $.

Denote $ F_i = H - H_i $. Then for $ t \in [0,1) $ we have
$$ \limsup_{h \rightarrow 0^+} \frac{1}{h} \int_{t}^{t+h} \|H(s,\cdot) - H(t,\cdot) \|_{\infty} ds $$
$$ \leqslant \left( \limsup_{h \rightarrow 0^+} \frac{1}{h} \int_{t}^{t+h} \| F_i(s,\cdot) - F_i(t,\cdot) \|_{\infty} ds \right)
 + \left( \limsup_{h \rightarrow 0^+} \frac{1}{h} \int_{t}^{t+h} \| H_i(s,\cdot) - H_i(t,\cdot) \|_{\infty} ds \right) $$
$$ = \limsup_{h \rightarrow 0^+} \frac{1}{h} \int_{t}^{t+h} \| F_i(s,\cdot) - F_i(t,\cdot) \|_{\infty} ds $$

$$ \leqslant \limsup_{h \rightarrow 0^+} \frac{1}{h} \int_{t}^{t+h} \| F_i(s,\cdot) \|_{\infty} + \| F_i(t,\cdot) \|_{\infty} ds $$ $$ = \| F_i(t,\cdot) \|_{\infty} + \limsup_{h \rightarrow 0^+} \frac{1}{h} \int_{t}^{t+h} \| F_i(s,\cdot) \|_{\infty} ds .$$ Denote $ f_i(t) := \| F_i(t,\cdot) \|_{\infty} $, we have $ f_i \in L^1([0,1]) $. By the standard Lebesgue differentiation theorem, for any $i$, we have $$ \lim_{h \rightarrow 0^+} \frac{1}{h} \int_{t}^{t+h} f_i(s) ds = f_i(t) ,$$ or $$ \lim_{h \rightarrow 0^+} \frac{1}{h} \int_{t}^{t+h} \| F_i(s,\cdot) \|_{\infty} ds = \| F_i(t,\cdot) \|_{\infty} $$ for almost every $ t \in [0,1) $. Therefore for any $ i $ we have
$$ \limsup_{h \rightarrow 0^+} \frac{1}{h} \int_{t}^{t+h} \| H(s,\cdot) - H(t,\cdot) \|_{\infty} ds
 \leqslant 2f_i(t) = 2 \| F_i(t,\cdot) \|_{\infty} $$ for almost every $ t \in [0,1) $.\\

 The sequence of functions, $f_i(t), \text{ }L^1$ converges to zero. Every $L^1$ converging sequence has a subsequence that converges almost everywhere.   Hence, by passing to a subsequence we may assume $f_i(t) $ converges to zero for almost every $ t \in [0,1) $.\\

\end{proof}

\begin{proof} [Proof of Lemma~\ref{L:reduction}]
 Because of Lemma~\ref{L:real-analysis}, $ H $ can be represented by a function $ H : [0,1] \times M \rightarrow \mathbb{R} $, such that for any $ t \in [0,1) $, the function $ H(t,\cdot) \in C_0(M) $ is continuous, and moreover for almost any $ t \in [0,1) $ we have $$ \lim_{ h \rightarrow 0^+} \frac{1}{h} \int_{t}^{t+h} \| H(s,\cdot) - H(t,\cdot) \|_{\infty} ds = 0. $$ Consider such a value of $ t \in [0,1) $. Take $ N \in \mathbb{N} $ large enough. Applying Lemma~\ref{L:basic-properties}, (3) for $ \alpha(s) = t + \frac{s}{N} $, we obtain a Hamiltonian $  G_{N}(s,x) = \frac{1}{N} H(t + \frac{s}{N},x) \in \mathcal{H}_{0} $. Applying Lemma~\ref{L:basic-properties}, (1), we get $ H_{N}(s,x) = N G_{N}(s,x) = H(t + \frac{s}{N},x) \in \mathcal{H}_{0} $. Denote $ h(x) = H(t,x) $ for $ x \in M $. We have $$ \int_{0}^{1} \| H_{N}(s,\cdot) - h(\cdot) \|_{\infty} ds =
N \int_{t}^{t + \frac{1}{N}} \| H(\tau,\cdot) - H(t,\cdot) \|_{\infty} d\tau \rightarrow_{N \rightarrow \infty} 0 ,$$ where we made the substitution $ \tau = t + \frac{s}{N} $. Therefore, because of Lemma~\ref{L:basic-properties}, (2), we have $ h \in \mathcal{H}_{0} $, and being time-independent, $ h \in \mathcal{H}^{st}_{0} $.
\end{proof}

\begin{proof} [Proof of Lemma~\ref{L:H^st}]
Let $ H \in \mathcal{H}^{st}_{0} $, and assume, for a contradiction, that $ H $ is a non-zero function. Let us show, that
then there exists a non-zero function $ h(x) \in \mathcal{H}^{st}_{0} \cap C^{\infty}(M) $.
First, there exists a point in $ M $ such that $ H $ is not constant in any neighborhood of it
(otherwise $ H $ is locally constant, and since $ M $ is connected, $ H $ is a constant function). Take such a point
$ x_0 $, and consider an open neighborhood $ x_0 \in U $, such that $ U \varsubsetneq M $ is moreover a Darboux chart.
Take $ y_0 \in U $, such that $ H(x_0) \neq H(y_0) $. There exists $ \Phi \in \Ham_{U}(M,\omega) $, such that
$ \Phi(x_0) = y_0 $. Define $ K = H \circ \Phi - H $. Then $ K \in H_{0}^{st}$, because of Lemma~\ref{L:basic-properties}
(4), (5). Moreover $ K $ is a non-zero function, and $\supp (K) \subset U $. Consider the $ L^\infty $ - closure $ \mathcal{L} $ of
the linear span of all functions of the form $ \Phi^{*} K $, where $ \Phi \in \Ham_U(M,\omega) $. In view of
Lemma~\ref{L:basic-properties} (2), (4), (5), we have $ \mathcal{L} \subset \mathcal{H}^{st}_{0} $. Let us show,
that $ \mathcal{L} $ contains a non-constant smooth function. Since $ U $ is a Darboux neighborhood, and the latter
statement has a local nature, we can further assume, that $ U \subset (\mathbb{R}^{2n},\omega_{std}) $, and moreover
we have $ K: U \rightarrow \mathbb{R} $ with $ K \neq 0 $, and moreover $ K=0 $ near $ \partial U $. Extend $ K $ as a
function $ K: \mathbb{R}^{2n} \rightarrow \mathbb{R} $ by $ 0 $ outside $ U $.
In this new situation, where we replaced the manifold $ M $ by $ \mathbb{R}^{2n} $, we keep the notation $ \mathcal{L} $
for the $ L^\infty $ - closure of the linear span of all functions of the form $ \Phi^{*} K $, where
$ \Phi \in \Ham_{U}(\mathbb{R}^{2n},\omega_{std}) $. For $ v \in \mathbb{R}^{2n} $, we denote $ K_{v}(x) = K(x-v) $.
Let us show, that when the norm $ \| v \| $ is small enough, we have $ K_{v} \in \mathcal{L} $.
Take a neighborhood $ W $ of $ \supp(K) $, such that $ \overline{W} \subset U $. Pick a function
$ \phi : \mathbb{R}^{2n} \rightarrow \mathbb{R} $, such that $ \supp(\phi) \subset U $ and moreover
$ \phi \equiv 1 $ on $ W $. For any $ v \in \mathbb{R}^{2n} $ define a Hamiltonian
$ G_{v} : \mathbb{R}^{2n} \rightarrow \mathbb{R} $ as $ G_{v}(x) = \omega_{std}(v , x) \phi(x) $ for $ x \in \mathbb{R}^{2n} $.
Then, for small $ \| v \| $, the time one map of its Hamiltonian flow coincides on $ \supp(K) $ with the
translation $ x \mapsto x + v $. Therefore we will have $ K_v = \big( \Phi^{-1}_{G_v} \big)^{*} K $, and
hence $ K_v \in \mathcal{L} $. Here we denote by $  \Phi^{t}_{G_v} $ the Hamiltonian flow of $ G_v $,
for $ t \in \mathbb{R} $.

Therefore we have shown, that $ K_v \in \mathcal{L} $ for small $ \| v \| $. As a conclusion, we have that for a
smooth function $ \chi $ with support lying in a sufficiently small neighborhood of $ 0 $, we have that the convolution
$ K \ast \chi $ lies in $ \mathcal{L} $ as well. We see this from the fact, that $ K \ast \chi $ is an $ L^\infty $ limit of a
sequence of finite sums $ \sum_{k=1}^{m} c_k K_{v_k} $, coming from the approximation of the Riemann integral by Riemann
sums (for the sake of completeness, we provide a detailed proof of a slightly more general fact in
Lemma~\ref{L:approximation-Riemman-sums} below). But of course, the function $ K \ast \chi $ is smooth, provided that the
function $ \chi $ is smooth. Moreover, $ K $ is a non-zero function on $ U $. Choose a sequence $ \chi_k $ of smooth
mollifiers approximating the $ \delta_0 $ - function, having sufficiently small supports. Then, we have
$ K \ast \chi_k \rightarrow K $ in the $ L^\infty $ topology, and hence the function $ K \ast \chi_k $ is a non-zero function too when $ k $ is large. This shows, that $ \mathcal{L} $ contains a non-zero smooth function. Therefore we conclude, that the space $ \mathcal{H}^{st}_{0} \cap  C^{\infty}(M) $ contains a non-zero smooth function, which contradicts Lemma~\ref{L:smooth-case}.

\end{proof}

\begin{lem} \label{L:approximation-Riemman-sums}
Let $ f,g : \mathbb{R}^d \rightarrow \mathbb{R} $ be two functions, where $ g \in L^1(\mathbb{R}^d) $ has compact support, and
$ f \in C_{c}(\mathbb{R}^{d}) $ is continuous with compact support.
Then, there exists a sequence of measures
$ \mu_k = \sum_{j=1}^{N_k} c_{kj} \delta_{v_{kj}} $, where $ c_{kj} \in \mathbb{R} $, and $ v_{kj} \in \mathbb{R}^{d} $,
such that $$ f \ast g = (L^\infty) \lim_{k \rightarrow \infty} f \ast \mu_k = (L^\infty) \lim_{k \rightarrow \infty}
\sum_{j=1}^{N_k} c_{kj} f(x-v_{kj}) .$$
Moreover, if for some open set $ V \subset \mathbb{R}^d $ we have $ \supp(g) \subset V $, then one can choose the measures
$ \mu_k $ to satisfy $ \supp(\mu_k) \subset V $.

Here, by the $ L^\infty $ norm on $ C_{c}(\mathbb{R}^d) $ we mean $$ \| h \|_{\infty} = \max_{x \in \mathbb{R}^d} | h(x) | .$$
\end{lem}

\begin{proof}
Given a measure $\mu$ we denote by $\| \mu \|$ its total variation.  Recall that if $\mu$ is absolutely continuous (with respect to the Liouville measure), then the total variation norm of $\mu$ coincides with the $L^1$ norm of its Radon-Nikodym derivative.
First of all, the function $ g $ can be approximated up to any precision in $ L^1 $ norm, by a function of the form
$ \mu \ast \psi $, where $ \mu $ is a measure of the form $ \mu = \sum_{j=1}^N c_j \delta_{v_j} $, $ \psi = \frac{1}{Vol(K_{\epsilon})} \chi_{K_\epsilon} $, $ K_\epsilon = [0,\epsilon]^n $ is a cube, and
$ \chi_{K_\epsilon} $ is the characteristic function of $ K_\epsilon $.  Moreover, we require  $ \epsilon $ to be
arbitrarily small, $ \| \mu \|  = \sum_{j=1}^N |c_j| \leqslant \| g \|_{L^1} + 1 $, and $ \supp(\mu) \subset V $.
To see this, observe that $ g $ can be approximated in $ L^1(\mathbb{R}^d) $, up to any precision, by a continuous function with support lying in $ V $, and
any continuous function $ h : \mathbb{R}^d \rightarrow \mathbb{R} $ with $ \supp(h) \subset V $ can be approximated in
$ L^1(\mathbb{R}^d) $ by $ \mu \ast \psi $, where $ \psi =  \frac{1}{Vol(K_{\epsilon})} \chi_{K_\epsilon} $, and
$$ \mu = \sum_{\alpha \in \epsilon \mathbb{Z}^n } Vol(K_{\epsilon}) h(\alpha) \delta_{\alpha} .$$
Note that $ \| \mu \| \rightarrow \| h \|_{L^1} $, as $ \epsilon \rightarrow 0 $, and $ \supp(\mu) \subset V $ for small $ \epsilon $.

Therefore, there exists a sequence of measures $ \mu_k = \sum_{j=1}^{N_k} c_{kj} \delta_{v_{kj}} $, and functions
$ \psi_k = \frac{1}{Vol(K_{\epsilon_k})} \chi_{K_{\epsilon_k}} $, such that
$$ g = (L^1) \lim_{k \rightarrow \infty} \mu_k \ast \psi_k ,$$ $$ \lim_{k \rightarrow \infty} \epsilon_k = 0 ,$$
$ \| \mu_k \| $ is uniformly bounded in $k$, and $ \supp(\mu_k) \subset V $ for all $ k \in
\mathbb{N} $.
 We have
$$ f \ast \mu_k  - f \ast g = (f \ast \mu_k - f \ast \mu_k \ast \psi_k) + (f \ast \mu_k \ast \psi_k - f \ast g)  $$
$$ = (f \ast \mu_k - f \ast \psi_k \ast \mu_k) + (f \ast \mu_k \ast \psi_k - f \ast g) $$ $$ =
(f  - f \ast \psi_k) \ast \mu_k + f \ast ( \mu_k \ast \psi_k - g) .$$ Therefore, we have
$$ \| f \ast \mu_k  - f \ast g \|_{\infty} \leqslant
\| f  - f \ast \psi_k \|_{\infty} \| \mu_k \| + \| f \|_{\infty} \| \mu_k \ast \psi_k - g \|_{L^1} .$$

Note that $ \psi_k \geqslant 0 $, $ \int \psi_k = 1 $, and $ supp(\psi_k) $ converges to 0, as $ k \rightarrow \infty $
(i.e., they are contained in any chosen neighborhood of $0$, for large $k$).  These properties imply that
$ \lim_{k \rightarrow \infty} \| f  - f \ast \psi_k \|_{\infty} = 0 $.  Moreover, $ \| \mu_k \| $ is a bounded sequence;
hence, $\| f  - f \ast \psi_k \|_{\infty} \| \mu_k \|$ converges to $0$, as $ k \rightarrow \infty $.
Also, we know that $$ \lim_{k \rightarrow \infty} \| \mu_k \ast \psi_k - g \|_{L^1}  = 0 .$$

We conclude that $ \lim_{k \rightarrow \infty} \| f \ast \mu_k  - f \ast g \|_{\infty} = 0 $.

\end{proof}

\begin{proof} [Proof of Theorem~\ref{T:main}]
Assume that $ H \in L_0^{(1,\infty)}([0,1] \times M) $, such that $ (\Id^t,H) \in \PHameo(M,\omega) $. Then $ H \in \mathcal{H}_{0} $, and Lemma~\ref{L:reduction} implies that for almost any $ t \in [0,1] $, the time-independent Hamiltonian $ h(x) = H(x,t) $
lies inside $ \mathcal{H}^{st}_{0} $. Then, for such values of $ t $, the function $ H( \cdot,t) $ is zero, by Lemma~\ref{L:H^st}. Therefore $ H = 0 $ in $ L_0^{(1,\infty)}([0,1] \times M) $.
\end{proof}

\section{Local uniqueness} \label{S:Local-uniqueness}
In this section, we present a generalization of Theorem 1.3 from~\cite{Oh-local}, to the $ L^{(1,\infty)} $ case. As an application we give an example of a continuous function which fails to be a generator of any topological Hamiltonian flow. As another application, we give an example of a continuous flow of homeomorphisms, which is a $C^0$ limit of smooth Hamiltonian flows, but is not a topological Hamiltonian flow.

\subsection{Local uniqueness for topological Hamiltonian functions} \label{SS:local-Hamiltonian-function}

 The uniqueness result from Theorem~\ref{T:main} admits a generalization, which is a local analog of it.
The following result holds.

\begin{thm} \label{T:local-result}
  Let $ (\Phi^t,H) \in \PHameo(M,\omega) $, and assume that the flow $ \Phi^t $ equals to the identity flow on some open
subset $ U \subset M $, i.e. for any $ x \in U $ and $ t \in [0,1] $ we have $ \Phi^t (x) = x $. Then for almost all
$ t \in [0,1] $, the restriction $ H(t,\cdot)|_{U} $ is a constant function.
\end{thm}
\begin{proof}[Proof of Theorem~\ref{T:local-result}]
 Let $ \Psi \in \Ham_U(M,\omega) $. Then we have $ \Psi^{-1} \circ \Phi^t \circ \Psi = \Phi^t $ for any $ t \in [0,1] $. On
the other hand, the Hamiltonian function of the flow $ \Psi^{-1} \circ \Phi^t \circ \Psi $ equals to $ \Psi^* H $, while
the Hamiltonian function of the flow $ \Psi^t $ equals $ H $. We can apply the uniqueness result for the Hamiltonian
function, corresponding to a topological Hamiltonian flow, which follows from Theorem~\ref{T:main}. We conclude
that $ H(t,\Psi(x)) = H(t,x) $ in $ L_0^{(1,\infty)}([0,1] \times M) $, for any $ \Psi \in \Ham_U(M,\omega) $. Let us
derive the result of the theorem from this. Choose a dense countable subset of $ U $, $ X = \{ x_0,x_1,x_2... \} \subset U $.
For every $ i \in \mathbb{N} $ pick some $ \Psi_i \in \Ham_U(M,\omega) $ satisfying $ \Psi_i(x_0) = x_i $. Then for
each $ i \in \mathbb{N} $ there exists a zero-measurable set $ S_i \subset [0,1] $, such that $ H(t,\Psi_i(x)) = H(t,x) $ for
any $ t \notin S_i $ and $ x \in M $. In particular $ H(t,x_i) = H(t,x_0) $ for any $ t \notin S_i $.
Denote $ S = \bigcup_{i=1}^{\infty} S_i $. Then $ S \subset [0,1] $ is of measure $ 0 $, and moreover we
have $ H(t,x_i) = H(t,x_0) $ for any $ t \notin S $. Fix arbitrary $ t \notin S $. The function $ H(t,\cdot) $ is
continuous on $ M $, and we have $ H(t,x) = H(t,x_0) $ for any $ x \in X $, while $ X \subset U $ is a dense
subset. We conclude that $ H(t,\cdot)|_{U} = const $ for any $ t \notin S $.
\end{proof}

\subsection{Local uniqueness for topological Hamiltonian flows} \label{SS:local-Hamiltonian-flow}

 \begin{thm}  \label{T:Local-uniqueness-of-flows} $ $
\begin{enumerate}

\item Let $ H \in L_0^{(1,\infty)}([0,1] \times M ) $ be a topological Hamiltonian function, that generates a topological Hamiltonian
flow $ \Phi^{t}_{H} $. Let $ U \subset M $ be an open subset. Assume that for almost all $ t \in [0,1] $, the restriction  $ H(t,\cdot)|_{U}$ is a constant function, say $c(t)$. Then $\Phi^{t}_{H}(x) = x $ for any $ x \in U $, $t \in [0,1] $.

\item Let $ H,K \in L_0^{(1,\infty)}([0,1] \times M ) $ be two topological Hamiltonian functions, that generate topological Hamiltonian
flows $ \Phi^{t}_{H}, \Phi^{t}_{K} $ respectively. Let $ U \subset M $ be an open subset. Assume that for any  $ t \in [0,1] $ we
have $ H(t,x) = K(t,x) \text{ } \forall x \in \Phi^{t}_{H}(U) $. Then we have $ \Phi^{t}_{H}(x) = \Phi^{t}_{K}(x) $ for
any $ x \in U $, $ t \in [0,1] $.

\end{enumerate}
\end{thm}

The proof of Theorem~\ref{T:Local-uniqueness-of-flows} (1) is similar to the proof of Theorem 3.1 from~\cite{Oh-local}.

\begin{proof} [Proof of Theorem~\ref{T:Local-uniqueness-of-flows}] $ $

(1): We know that there exists a sequence of smooth Hamiltonians $H_i, \text{ } L^{(1,\infty)} $ converging to $H$ whose
flows $\Phi^{t}_{H_i} \text{ } C^0 $ converge to $ \Phi^t_H $.  For a given point $x \in U$, pick a neighborhood of it
$V$ which is compactly contained in $U$, and take a smooth cut off function $\beta$ such that support of $\beta$ is contained
in $U$ and $\beta=1 \text{ on } V .$ For any $ i \in \mathbb{N} $, for any $ t \in [0,1] $, define
$$ c_i(t) = \frac{\int_U H_i(t,x) \omega^n}{\int_U \omega^n} ,$$
$$ d_i(t) = \frac{ \int_{M} \beta(x) (H_i(t,x) - c_i(t)) \omega^n }{ \int_{M} \omega^n } ,$$ and then define new smooth
normalized Hamiltonians $$ G_i(t,x)= \beta(x) (H_i(t,x) - c_i(t)) - d_i(t) .$$ Then
$G_i(t,x)= H_i(t,x) - c_i(t) - d_i(t)$ on $V$. It is easy to see that $(L^{(1,\infty)}) \lim G_i = 0  $. Assume for
a contradiction, that for some $ t \in [0,1] $ we have $ \Phi^{t}_{H}(x) \neq x $. Then we can find some $ 0 < T \leqslant 1 $ such that
$ \Phi^{T}_{H}(x) \neq x $ and moreover $ \Phi^{t}_{H}(x) \in V $ for all $ t \in [0,T] $. Therefore,
since $ (C^0)\lim \Phi^{t}_{H_i} = \Phi^{t}_{H} $, there exists a small enough open neighborhood $ W $
of $ x $, $ x \in W \subset V $, such that $ \Phi^{T}_{H_i}(W) \cap W = \emptyset $ and
moreover $ \Phi^{t}_{H_i}(W) \subset V $ for all $ t \in [0,T] $, for sufficiently large $i$.
Because $G_i(t,x) = H_i(t,x) - c_i(t) - d_i(t)$ on all of $V$, we have $ \Phi^{T}_{G_i}(W) \cap W = \emptyset $ as
well, for $ i $ large enough. Then the energy-capacity inequality implies that $\| G_i \|_{(1,\infty)}$ is
bounded from below by the displacement energy of $W$, which is known to be positive. But we know that $G_i$
$L^{(1,\infty)}$-converges to $0$, and this is a contradiction.

(2): Consider the flow $(\Phi^t_H)^{-1}\circ \Phi^t_K$.  This flow is generated by the Hamiltonian $ \overline{H} \# K(t,x) =  -H(t,\Phi^t_{H}(x)) + K(t,\Phi^t_{H}(x))$. By our assumption, this Hamiltonian is zero on $U$. Therefore, by (1) we obtain $ (\Phi^t_H)^{-1}\circ \Phi^t_K(x) = x $, and hence, $ \Phi^t_H(x) = \Phi^t_K(x) $
for any $ x \in U $, $ t \in [0,1] $.

\end{proof}

\subsection{Example of a non-generator}
We will now construct an example of a continuous function which does not generate a topological Hamiltonian flow.
Let $ (M, \omega) $ be a closed symplectic manifold. Consider some Darboux chart $ W \subset M $, endowed with symplectic coordinates $ (x_1,y_1,...,x_n,y_n) $, and assume
for simplicity that $$ 0 = (0,0,...,0,0) \in W .$$ Take any continuous function $ K : M \rightarrow \mathbb{R} $, such that for  every point $$ (x_1,y_1,...,x_n,y_n) \in W,$$ sufficiently close to $ 0 \in W $, we have $ K(x_1,y_1,...,x_n,y_n) = |x_1| $. Let us show that such a function does not generate a topological Hamiltonian flow.
Assume for a contradiction, that $ K $ generates a topological Hamiltonian flow $ \Phi^{t}_{K} $ on  $ M $. There exists $ \epsilon > 0 $, such that we have
$ (x_1,y_1,...,x_n,y_n) \in W $ and $$ K(x_1,y_1,...,x_n,y_n) = |x_1| ,$$ provided $ |x_i|,|y_i| \leqslant \epsilon $, $ i = 1,2,...,n $.
Consider any smooth function $ \phi : M \rightarrow \mathbb{R} $ supported in $ W $, such that $ \phi(x) = 1 $ for $ x = (x_1,y_1,...,x_n,y_n) \in W $ with $ |x_i|,|y_i| \leqslant \epsilon $, $ i = 1,2,...,n $.

Define $ H_1 : M \rightarrow \mathbb{R} $ as $ H_1(x) = x_1 \phi(x) $, for $ x = (x_1,y_1,...,x_n,y_n) \in W $, and as $ H_1(x) = 0 $ for $ x \in M \setminus W $.
Define $ U_1 \subset W $ to be the set of all $ (x_1,y_1,...,x_n,y_n) \in W $, such that $ 0 < x_1 < \epsilon $, $ |y_1| < \frac{\epsilon}{2} $, and $ |x_i|,|y_i| < \epsilon $ for $ i = 2,3,...,n $. Apply Theorem~\ref{T:Local-uniqueness-of-flows} (2), to $ H_1, K, U_1 $ in the time interval $ [0,\frac{\epsilon}{2}] $ (of course, the time interval [0,1] in Theorem~\ref{T:Local-uniqueness-of-flows} can be replaced by any other time interval). We conclude that $$ \Phi^t_K(x_1,y_1,...,x_n,y_n) = \Phi^t_{H_{1}}(x_1,y_1,...,x_n,y_n) = (x_1,y_1-t,...,x_n,y_n) $$ for any $ 0 \leqslant t \leqslant \frac{\epsilon}{2} $, for any $ (x_1,y_1,...,x_n,y_n) \in W $, provided $ 0 < x_1 < \epsilon $, $ |y_1| < \frac{\epsilon}{2} $, $ |x_i|,|y_i| < \epsilon $ for $ i = 2,3,...,n $.

Now define $ H_2 : M \rightarrow \mathbb{R} $ as $ H_2(x) = - H_1(x) $, and let $ U_2 \subset W $ be the set of all $ (x_1,y_1,...,x_n,y_n) \in W $, such that $ -\epsilon < x_1 < 0 $, $ |y_1| < \frac{\epsilon}{2} $, and $ |x_i|,|y_i| < \epsilon $ for $ i = 2,3,...,n $. Applying Theorem~\ref{T:Local-uniqueness-of-flows} (2), to $ H_2, K, U_2 $ in the time interval $ [0,\frac{\epsilon}{2}] $, in a similar way we obtain $$ \Phi^t_K(x_1,y_1,...,x_n,y_n) = \Phi^t_{H_{2}}(x_1,y_1,...,x_n,y_n) = (x_1,y_1+t,...,x_n,y_n) $$ for any $ 0 \leqslant t \leqslant \frac{\epsilon}{2} $, for any $ (x_1,y_1,...,x_n,y_n) \in W $, provided $ -\epsilon < x_1 < 0 $, $ |y_1| < \frac{\epsilon}{2} $, $ |x_i|,|y_i| < \epsilon $ for $ i = 2,3,...,n $.

Clearly, such flow $ \Phi^t_K $ is not a flow of homeomorphisms, and we come to a contradiction.

\subsection{Example of a non-flow}

In this section, for any closed symplectic manifold $ (M^{2n},\omega) $, we construct a continuous flow of homeomorphisms, i.e. a continuous path in the group $Homeo(M,\omega)$, which is a $C^0$ limit of smooth Hamiltonian flows, but is not a topological Hamiltonian flow. This flow fails to be a topological Hamiltonian flow, because there exist no $H \in L_0^{(1,\infty)}([0,1] \times M)$ generating the flow.

The following example is a generalization of the one considered by Oh and M\"uller (see~\cite{OM}, and~\cite{Mu} Section 2.4.1).
Let $ (M^{2n},\omega) $ be a closed $2n$-dimensional symplectic manifold. Consider a smooth symplectic embedding of
a small ball $ i : (B^{2n}(a),\omega_{std}) \hookrightarrow (M,\omega) $, and denote $ V = i(B^{2n}(a)) $. Consider
the Darboux coordinates $ (x_1,y_1,...,x_n,y_n) $ on $ V $ coming from $ B^{2n}(a) $.
For a smooth function $ h : (0,a) \rightarrow \mathbb{R} $, which is zero near $a$, define a Hamiltonian
$ H : M \setminus \{ i(0) \} \rightarrow \mathbb{R} $, such that for $ x = (x_1,y_1,...,x_n,y_n) \in V $ we have $ H(x) = h(r) $ where $ r = \sqrt{x_1^2 + y_1^2 + ... + x_n^2 + y_n^2} $, and for $ x \in M \setminus V $ we have $ H(x) = 0 $. Then $ H $ has a well defined smooth Hamiltonian flow $ \Phi^t : M \setminus \{ i(0) \} \rightarrow M \setminus \{ i(0) \} $, and we can extend $ \Phi^t $ to a continuous flow on $ M $, by setting $ \Phi^t(i(0)) = i(0) $.
Moreover,  in the case when $ h(r) = 0 $ for small $ r $, the flow $ \Phi^t $ is Hamiltonian, where the (un-normalized) Hamiltonian function equals $ H $ on $ M \setminus \{ i(0) \} $, and equals $0$ at $ i(0) $. We say that $ \Phi^t $ is the rotation associated to $ h $.

Now, consider a smooth function $ f : (0,a) \rightarrow \mathbb{R} $, such that $ f(r) = \frac{1}{r} $ for $ r \in (0,\frac{a}{3}) $, and also $ f(r) = 0 $ for $ r \in (\frac{2a}{3},a) $. Let $ \Psi^t : M \rightarrow M $ be the rotation, associated to $ f $. Then the flow $ \Psi^t $ is a $C^0$-limit of smooth Hamiltonian flows. Indeed, take a sequence of smooth functions $ f_n : (0,a) \rightarrow \mathbb{R} $, such that $ f_n(r) = 0 $ for $ r \in (0, \frac{1}{n}) $, $ f_n(r) = f(r) $ for $ r \in (\frac{2}{n},a) $, and for each $ n $ define $ \Psi_n^t $ to be the rotation associated to $ f_n $. Then $ \Psi_n^t $ is the needed sequence of smooth Hamiltonian flows.

Assume, for a contradiction, that $ \Psi^t $ is in fact a topological Hamiltonian flow.
Then denote by  $ H(t,x) $ its Hamiltonian function. Take $ f_n $, $ \Psi_n^t $ as above, and denote by $ H_n(x) $ the
normalized Hamiltonian function of $ \Psi_n^t $. We obtain that the flow $ ( \Psi_n^t )^{-1} \circ \Psi^t $ is generated by $ K_n(t,x) = -H_n(\Psi_n^t(x)) + H(t,\Psi_n^t(x)) $.
Moreover, we have $ ( \Psi_n^t )^{-1} \circ \Psi^t = Id^t $ on
$$ V_n := \left\{ x = (x_1,y_1,...,x_n,y_n) \in V \, \Big| \, \, r = \sqrt{x_1^2 + y_1^2 + ... + x_n^2 + y_n^2} > \frac{2}{n} \right\} .$$ Then, from Theorem~\ref{T:local-result} we have $ K_n(t,x) = -H_n(\Psi_n^t(x)) + H(t,\Psi_n^t(x)) = c_n(t) $ for almost all $ t $, for $ x \in V_n $. Since $ \Psi_n^t(V_n) = V_n $, we get $ H(t,x) = H_n(x) + c_n(t) $ for almost all $ t $, for $ x \in V_n $. This immediately implies that for almost any fixed $ t $, the function $ H(t,\cdot) $ is unbounded, and we come to a contradiction.

\subsection*{Acknowledgments}
The first author was supported by the Mathematical Sciences Research Institute.
The second author was partially supported by NSERC fellowship CGS D3-332749-2007.
We are thankful to the organizers of the 2009-10 program in Symplectic and
Contact Geometry and Topology, and to the staff of the Mathematical Sciences Research Institute at Berkeley,
for providing an enjoyable and lively research atmosphere. We would like to thank Octav Cornea, Yakov Eliashberg,
Boris Ettinger, Baoping Liu, Stefan M\"uller, Yong-Geun Oh, Fraydoun Rezakhanlou and Alan Weinstein for helpful
discussions and comments.  We thank an anonymous referee for several helpful comments and suggestions.

\bigskip
\noindent Lev Buhovski\\
School of Mathematical Sciences, Tel Aviv University \\
{\it email:} levbuh@post.tau.ac.il

\bigskip
\noindent Sobhan Seyfaddini \\ 
Department of Mathematics and their Applications \\ \'Ecole Normale Sup\'erieure \\
{\it email:} sobhan.seyfaddini@ens.fr

\end{document}